\documentclass[12pt,a4paper,reqno]{amsart}
\usepackage{amssymb,amsfonts,mathrsfs}
\usepackage[T1]{fontenc}
\usepackage[utf8]{inputenc}
\usepackage[english]{babel}
\usepackage{enumerate}
\usepackage{color}
\usepackage{graphicx}

\advance\textwidth30mm \advance\hoffset-14mm
\advance\textheight30mm \advance\voffset-18mm
\newtheorem{thm}{Theorem}
\newtheorem{lem}{Lemma}
\newtheorem{cor}{Corollary}

\theoremstyle{remark}

\theoremstyle{definition}

\newcommand\sign{\operatorname{sign}}
\renewcommand\Re{\operatorname{Re}}

\newcommand\supp{\operatorname{supp}}

\newcommand\Cdot{{\mskip2mu\cdot\mskip2mu}}
\newcommand\PW{\mathit{PW}}
\newcommand\res{\operatorname*{res}}
\newcommand\pv{\operatorname{p.v.}}

\begin{document}

\title{The Nikolskii Constant in Arbitrary Dimension}
\author{D.\,V.~Gorbachev}
\address{Saint Petersburg State University}
\email{dvgmail@mail.ru}
\thanks{The study was carried out with the financial support of the Ministry of Science
and Higher Education of the Russian Federation in the framework of a scientific
project under agreement No.~075-15-2025-013.}
\date{}

\begin{abstract}
We study the radial extremal function $\varphi(|\Cdot|)$ arising in the problem
of finding the sharp Nikolskii constant $\mathcal C_d$ in
$\mathit{PW}_1^1(\mathbb R^d)$ for arbitrary dimension $d\ge1$. We prove a factorization
$\varphi=\Phi_1\Phi_2$, where $\Phi_1$ and $\Phi_2$ are entire functions of
exponential type $1/2$ satisfying a functional equation and second-order
differential equations with polynomial coefficients. As a result, the original extremal problem is reduced to a one-dimensional
spectral problem depending on at most $d+1$ parameters. We also obtain a zeta interpretation of the coefficients of the
polynomial appearing in the functional equation and a multiplicative
equilibrium condition for the zeros of the extremal function. These results can be used to construct several algorithms for computing
$\mathcal C_d$.
\end{abstract}



\maketitle

\section{Introduction}

\subsection{Statement of the problem and the extremal function}
Let $\PW_r^p(\mathbb R^d)$ be the Paley--Wiener class of functions of
spherical exponential type at most~$r$ that belong to $L^p(\mathbb R^d)$,
$d\in\mathbb N$. Consider the extremal problem for the sharp Nikolskii
constant
\begin{equation}\label{c-prob}
\mathcal C_d^{-1}=\inf\{\|f\|_1\colon f\in\PW_1^1(\mathbb R^d),\ f(0)=1\}.
\end{equation}
This problem has a long history; see, for example, the survey \cite{Go21}. A
substantial advance for $d=1$ was made in \cite{Bo25}. In \cite{Go26}, we
extended the results of \cite{Bo25} to odd dimensions.

We will make frequent use of the notation and results of \cite{Go26}. In
particular, the problem~\eqref{c-prob} has a unique radial extremal function
$\varphi_d(|\Cdot|)$; see also \cite{Da21}. For a fixed $d$, write
$\varphi=\varphi_d$, and denote its positive zeros by
$0<\tau_1<\tau_2<\ldots,$ $\tau_k=\tau_{k,d}$. The function~$\varphi$ is an
even entire function of exact exponential type~$1$, bounded in modulus by one
on $\mathbb R$, and has the canonical product
\begin{equation}\label{phi-product}
\varphi(z)=\prod_{k=1}^{\infty}
\Bigl(1-\frac{z^2}{\tau_k^2}\Bigr),\quad z\in\mathbb C,
\end{equation}
where $\tau_k\sim\pi k$ as $k\to\infty$ for fixed $d$.

\subsection{Extremality condition}
Let $v_{d}$ be the volume of the unit ball $B_1^d\subset\mathbb R^d$ and
$\omega_{d-1}=dv_{d}$ the surface area of the unit sphere $\mathbb S^{d-1}$.
Set
\[
a_d=\frac1{2v_{d}\mathcal C_d}.
\]
The solution of \eqref{c-prob} is based on a variational extremality condition,
which in one-dimensional form reads
\begin{equation}\label{extremality-1D}
g(0)=\frac{\mathcal C_d\omega_{d-1}}2 \int_{\mathbb
R}\sigma(x)g(x)|x|^{d-1}\,dx,
\end{equation}
where $g$ is an arbitrary even entire function of exponential type at most~$1$
such that $|x|^{d-1}g(x)\in L^1(\mathbb R)$, and
\[
\sigma(x)=\sign\varphi(x),\quad x\in\mathbb R.
\]
Instead of \eqref{extremality-1D}, we will use its consequence (see
Lemma~\ref{lem:extremality-fourier}), which in $\mathcal S'(\mathbb R)$ has
the form
\begin{equation}\label{eq:extremality-fourier}
\widehat{|x|^{d-1}\sigma}=\frac{4a_d}{d}\quad \text{on $(-1,1)$},
\end{equation}
where
\[
\widehat f(t)=\int_{\mathbb R}f(x)e^{-itx}\,dx
\]
denotes the one-dimensional Fourier transform, extended to
$\mathcal S'(\mathbb R)$ by duality in the standard way.

\subsection{Known estimates}
It was proved in \cite{Da21} that
\[
2^{-d}\le \mathcal L^*(d)\le
{}_1F_2\Bigl(\frac d2;\frac d2+1,\frac d2+1;
-\frac{\beta_d^2}{4}\Bigr),
\]
where
\begin{equation}\label{def-Lstar}
\mathcal L^*(d)=\frac{(2\pi)^d}{v_{d}}\,\mathcal C_d
\end{equation}
is the normalized Nikolskii constant, and $\beta_d$ is the first positive zero
of the Bessel function $J_{d/2}$. Hence
\[
2^{-d}\le \mathcal L^*(d)\le \bigl(\sqrt{2/e}\,\bigr)^{d(1+o(1))},\quad
d\to\infty,
\]
where $\sqrt{2/e}=0.857\ldots.$

For the first three dimensions, the numerical values of the normalized
constant and the corresponding general lower and upper bounds are
\[
\begin{array}{c|c|c|c|c}
d & \text{lower bound} & \mathcal L^*(d) & \text{upper bound} & \text{source}\\ \hline
1 & 0.5\phantom{00}
& 0.54092882190183058939\ldots
& 0.589\ldots
& \text{\cite{Bo25}}\\
2 & 0.25\phantom{0}
& 0.28214506418970912412\ldots
& 0.382\ldots
& \text{present paper}\\
3 & 0.125
& 0.14502922550655639263\ldots
& 0.261\ldots
& \text{\cite{Go26}}
\end{array}
\]

\subsection{From odd dimensions to a general proof}
In \cite{Go26}, the oddness of the dimension was used essentially. If
$d=2m+1$, then
\[
|x|^{d-1}\sigma(x)=x^{2m}\sigma(x),
\]
so the left-hand side of \eqref{eq:extremality-fourier} is expressed through
a finite-order derivative of the distribution $\widehat\sigma$. It follows
that the restriction of $\widehat\sigma$ to $(-1,1)$ coincides with a
polynomial.

For even $d=2m$, the weight $|x|^{d-1}=|x|^{2m-1}$ is no longer a polynomial,
and a direct use of the same argument leads to a nonlocal operator on the
Fourier side. The key observation is to replace $\sigma$ by
$\tilde\sigma(x)=\sign(x)\sigma(x)$. Then
\[
|x|^{2m-1}\sigma(x)=x^{2m-1}\tilde\sigma(x)
\]
and an ordinary finite-order derivative appears again, now for
$\tilde\sigma$. After differentiating~$\tilde\sigma$, the Cauchy transform of
the resulting measure gives a meromorphic function whose boundary
values have one-sided Fourier spectra.

In the odd-dimensional proof, a similar meromorphic function was obtained
from the derivative of the analytic signature $S$, constructed by means of
the Hermite--Biehler theorem. The new approach shows that it can be obtained
directly as the Cauchy transform of the measure $\sigma'$. Thus the
Hermite--Biehler theory is not needed for the main spectral construction, and
the two cases now admit a unified proof.

\subsection{Main result}
To state the theorem, we need several definitions. Define the signature
\begin{equation}\label{def-sigma-tilde}
\tilde\sigma(x)=
\begin{cases}
\sigma(x), &\text{$d$ odd},\\ \sign(x)\sigma(x), &\text{$d$ even},
\end{cases}
\end{equation}
and put, in the sense of distributions,
\[
\nu=\tilde\sigma'.
\]
Then for every $d$
\begin{equation}\label{weight-sigma-tilde}
|x|^{d-1}\sigma(x)=x^{d-1}\tilde\sigma(x).
\end{equation}

Let
\[
\epsilon=
\begin{cases}
0, &\text{$d$ odd},\\ 1, &\text{$d$ even}.
\end{cases}
\]
Introduce the nonnegative integer
\[
N=\frac{d-1-\epsilon}{2}.
\]
We will show below (see Subsection~\ref{subsec-Pd}) that the restriction of the
distribution $2^{1-\epsilon}i^\epsilon\,\widehat{\tilde\sigma}$ to $(-1,1)$
coincides with the real polynomial
\begin{equation}\label{intro-Pd}
P_d(t)=\sum_{n=0}^Np_nt^{2n+\epsilon},\quad
p_N=(-1)^N2^{1-\epsilon}\,\frac{4a_d}{d!}.
\end{equation}
Using the coefficients $p_n$, define the even polynomial
\begin{equation}\label{intro-Rd}
R_d(z)=\frac\epsilon2\,z^{2(N+\epsilon)} +2^{\epsilon-3}\sum_{n=0}^N(-1)^n
(2n+\epsilon+1)!\,p_nz^{2(N-n)}= \sum_{j=0}^{N+\epsilon}r_jz^{2j},
\end{equation}
for which $r_0=R_d(0)=a_d$.

Set
\begin{equation}\label{def-lambda}
\lambda_q=\sum_{k=1}^{\infty}
\frac{(-1)^k}{\tau_k^{2q+\epsilon+1}}<0,\quad q\ge0.
\end{equation}
The series converges absolutely except when $\epsilon=q=0$, in which case it
converges by the alternating series test. Its negativity follows from the
strict increase of~$\tau_k$. Using the coefficients $r_j$ and the numbers
$\lambda_q$, define the polynomial
\begin{equation}\label{intro-Kd}
K_d(z)=\frac14\,z^{d+1}
+2\sum_{n=0}^N
\biggl(\sum_{j=0}^nr_j\lambda_{n-j}\biggr)z^{2n+\epsilon}.
\end{equation}

In the domain $\mathbb C\setminus(-\infty,0]$, set
\begin{equation}\label{intro-hd}
h_d(z)=\frac\epsilon2\,\log z -\sum_{j=0}^N\frac{r_j}{d-2j}\,z^{-(d-2j)},\quad
h_d'(z)=\frac{R_d(z)}{z^{d+1}},
\end{equation}
where $\log z$ denotes the principal branch of the logarithm.

Finally, define the factors $\Phi_j=\Phi_{j,d}$, $j=1,2$, of the extremal
function. For odd $d$, put
\begin{equation}\label{def-Phi12-odd}
\Phi_1(z)=\prod_{k=1}^{\infty}
\Bigl(1+(-1)^k\frac{z}{\tau_k}\Bigr),\quad
\Phi_2(z)=\Phi_1(-z),
\end{equation}
where consecutive factors in the product for $\Phi_1$ are grouped in pairs.
For even~$d$, put
\begin{equation}\label{def-Phi12-even}
\Phi_1(z)=\prod_{k=1}^{\infty}
\Bigl(1-\frac{z^2}{\tau_{2k-1}^2}\Bigr),\quad
\Phi_2(z)=\prod_{k=1}^{\infty}
\Bigl(1-\frac{z^2}{\tau_{2k}^2}\Bigr).
\end{equation}
In both cases,
\[
\varphi(z)=\Phi_1(z)\Phi_2(z),
\]
and the positive zeros of $\Phi_1$ and $\Phi_2$ are respectively
$\tau_1,\tau_3,\ldots$ and $\tau_2,\tau_4,\ldots.$ The functions~$\Phi_j$
are entire functions of finite exponential type (see
Lemma~\ref{lem-Phi12}). We also introduce
\begin{equation}\label{def-Phi12-tilde}
\tilde\Phi_1(z)=e^{-h_d(z)}\Phi_1(z),\quad
\tilde\Phi_2(z)=e^{h_d(z)}\Phi_2(z),\quad z\in \mathbb C\setminus(-\infty,0].
\end{equation}

We can now state the main result.

\begin{thm}\label{thm-main}
Let $d\ge1$, and use the notation introduced above.

\textup{(a)} For all $z\in\mathbb C$, the following functional identity holds:
\begin{equation}\label{intro-FE}
z^{d+1}\bigl(\Phi_1'(z)\Phi_2(z)-\Phi_1(z)\Phi_2'(z)\bigr)
-2R_d(z)\Phi_1(z)\Phi_2(z)=-2a_d.
\end{equation}

\textup{(b)} The functions $\tilde\Phi_1$, $\tilde\Phi_2$ are linearly
independent solutions of the equation
\begin{equation}\label{intro-common-ODE}
y''+\frac{d+1}{z}\,y'
+\Bigl(\frac{K_d(z)}{z^{d+1}}
-\frac{R_d(z)^2}{z^{2d+2}}\Bigr)y=0
\end{equation}
in the domain $\mathbb C\setminus(-\infty,0]$, and their Wronskian is
\[
W(\tilde\Phi_1,\tilde\Phi_2)= \tilde\Phi_1\tilde\Phi_2'
-\tilde\Phi_1'\tilde\Phi_2 =\frac{2a_d}{z^{d+1}}.
\]
Equivalently, the functions $\Phi_1$, $\Phi_2$ satisfy the system
\begin{equation}\label{intro-ODE}
\begin{cases}
z^{d+1}\Phi_1''+\bigl((d+1)z^d-2R_d\bigr)\Phi_1'
+\bigl(K_d-R_d'\bigr)\Phi_1=0,\\
z^{d+1}\Phi_2''+\bigl((d+1)z^d+2R_d\bigr)\Phi_2'
+\bigl(K_d+R_d'\bigr)\Phi_2=0.
\end{cases}
\end{equation}
\end{thm}

For the analysis of the problem, it is convenient to use the Liouville
transformation and write \eqref{intro-common-ODE} in the standard form without
the first derivative.

\begin{cor}\label{cor-liouville}
On the positive half-line, put
\[
u_j(x)=x^{(d+1)/2}\tilde\Phi_j(x),\quad j=1,2.
\]
Then $u_1$, $u_2$ are linearly independent solutions of
\begin{equation}\label{intro-Liouville-ODE}
u''+V_d(x)u=0,\quad x>0,
\end{equation}
where the potential is
\begin{equation}\label{intro-V}
V_d(x)=\frac{K_d(x)}{x^{d+1}} -\frac{R_d(x)^2}{x^{2d+2}} -\frac{d^2-1}{4x^2}.
\end{equation}
Their Wronskian is $W(u_1,u_2)=2a_d$.
\end{cor}

Equation \eqref{intro-Liouville-ODE} allows us to apply Sturm theory directly
and to study the asymptotic behavior of the factors $\Phi_j$. Together with
their other properties, this gives the following result.

\begin{cor}\label{cor-Phi-growth}
The functions $\Phi_1$, $\Phi_2$ have exact exponential type $1/2$ and
\begin{equation}\label{intro-decay}
\Phi_1(x)=O\bigl(|x|^{-(d+1-\epsilon)/2}\bigr),\quad
\Phi_2(x)=O\bigl(|x|^{-(d+1+\epsilon)/2}\bigr),
\quad |x|\to\infty.
\end{equation}
In particular,
\[
\varphi(x)=O\bigl(|x|^{-(d+1)}\bigr),\quad |x|\to\infty.
\]
\end{cor}

We also give an equation for the extremal function $\varphi$ itself.

\begin{cor}\label{cor-phi-ODE}
The extremal function $\varphi$ satisfies, for $z\ne0$, the equation
\[
\varphi'''+\frac{3(d+1)}z\,\varphi''
+\Bigl(4q_d+\frac{(d+1)(2d+1)}{z^2}\Bigr)\varphi'
+\Bigl(2q_d'+\frac{4(d+1)}z\,q_d\Bigr)\varphi=0,
\]
where
\[
q_d(z)=\frac{K_d(z)}{z^{d+1}}-\frac{R_d(z)^2}{z^{2d+2}}.
\]
\end{cor}

\subsection{The spectral problem}
The polynomial $R_d$ is completely determined by the coefficients
$p_0,\ldots,p_N$, while the polynomial $K_d$ depends on
$p_0,\ldots,p_N$ and the numbers $\lambda_0,\ldots,\lambda_N$. In turn, all
these $d+1-\epsilon$ quantities can be recovered from the extremal function
$\varphi$: the coefficients $p_n$ are obtained from the restriction of
$\widehat{\tilde\sigma}$ to $(-1,1)$, while the numbers $\lambda_q$ are given
by alternating power sums over the zeros $\tau_k$.

Thus the original extremal problem admits several equivalent one-dimensional
spectral formulations. For qualitative analysis, equation
\eqref{intro-Liouville-ODE} is especially convenient. For numerical
computations, the choice of the form depends on the dimension: for $d=1$ and
$d=3$, the problems obtained by applying the Fourier transform to the
equations for $\Phi_j$ are effective, whereas for $d=2$ it is convenient to
use the system \eqref{intro-ODE} directly and the resulting recurrence
relations for the Taylor coefficients. A convenient general algorithm can
be based on the equilibrium condition \eqref{a-tau}.

\subsection{Outline of the proof of Theorem~\ref{thm-main}}
Identity \eqref{weight-sigma-tilde} replaces the weight $|x|^{d-1}$ in the
extremality condition \eqref{eq:extremality-fourier} by the monomial
$x^{d-1}$ in every dimension. The derivative $\nu=\tilde\sigma'$ is a locally
finite discrete measure. Its Cauchy transform
\[
M(z)=\pv\int_{\mathbb R}\frac{d\nu(t)}{z-t}
\]
has boundary values $M^{\pm}(x)$ with one-sided Fourier spectra. The restriction of
$\widehat\nu$ to $(-1,1)$ is determined by the polynomial $P_d$, and the
special choice of the polynomial $R_d$ allows us to construct
\[
\Theta_d(z)=z^dM(z)-2\epsilon z^{d-1}+\frac{4R_d(z)}z,
\]
for which
\[
\supp\widehat{\Theta_d^+}\subset[1,\infty),\quad \supp\widehat{\Theta_d^-}\subset(-\infty,-1].
\]
The defect
\[
H_d=\varphi\Theta_d-\frac{4a_d}{z}
\]
is an entire function whose spectrum is concentrated at zero. Hence it is a
polynomial. The values of $H_d$ at the zeros of $\varphi$ and the
Plancherel--P\'olya inequality imply $H_d\equiv0$.~If
\begin{equation}\label{intro-Psi}
\Psi(z)=z^\epsilon\,\frac{\Phi_2(z)}{\Phi_1(z)},
\end{equation}
then $M=2\Psi'/\Psi$, and the identity $\varphi\Theta_d=4a_d/z$ becomes the
functional identity in part~(a).

To prove part~\textup{(b)}, we differentiate \eqref{intro-FE} and substitute
the zeros of $\Phi_1$ and $\Phi_2$. This shows that the meromorphic functions
\begin{align}
\mathcal P_1(z)&=
-\frac{z^{d+1}\Phi_1''+\bigl((d+1)z^d-2R_d\bigr)\Phi_1'}{\Phi_1},
\label{def-P1}\\
\mathcal P_2(z)&=
-\frac{z^{d+1}\Phi_2''+\bigl((d+1)z^d+2R_d\bigr)\Phi_2'}{\Phi_2}
\label{def-P2}
\end{align}
have no poles. Lemma~\ref{lem-logder} on the logarithmic derivative shows
that $\mathcal P_1$, $\mathcal P_2$ are polynomials. Differentiating
\eqref{intro-FE} gives
\[
\mathcal P_2-\mathcal P_1=2R_d',
\]
and comparison of the expansions at zero shows that their half-sum coincides
with the polynomial $K_d$ defined in \eqref{intro-Kd}. Thus
\[
\mathcal P_1=K_d-R_d',\quad \mathcal P_2=K_d+R_d',
\]
which gives the system \eqref{intro-ODE}.

\subsection{Organization of the paper}
Section~\ref{sec-aux} contains auxiliary statements.
Section~\ref{sec-functional} proves part~\textup{(a)} of
Theorem~\ref{thm-main}: the Cauchy transform of the discrete measure gives the
spectral gap and the functional identity. Section~\ref{sec-ode} proves
part~\textup{(b)} and constructs the general differential equation.
Section~\ref{sec-corollaries} proves the corollaries concerning the Liouville
transformation, exponential type, and the equation for the extremal function.
In Section~\ref{sec-d2}, the general scheme is written out in detail for the
first even dimension $d=2$, together with a numerical computation.
Section~\ref{sec-zeta} is devoted to the zeta interpretation of the
coefficients of the polynomial $R_d$, while
Section~\ref{sec-final} collects additional consequences and discusses the
asymptotics as $d\to\infty$.

\section{Auxiliary statements}\label{sec-aux}

\begin{lem}\label{lem:extremality-fourier}
For every $d\in\mathbb N$, in $\mathcal S'(\mathbb R)$,
\[
\widehat{|x|^{d-1}\sigma}=\frac{4a_d}{d}\quad \text{on $(-1,1)$}.
\]
\end{lem}

\begin{proof}
Let $\psi\in C_c^\infty((-1,1))$ be even and
\[
g(x)=\frac1{2\pi}\int_{\mathbb R}\psi(t)e^{ixt}\,dt.
\]
Then $g\in\mathcal S(\mathbb R)$ is even and has exponential type less
than~$1$. Hence, by \eqref{extremality-1D}, the identities
$\omega_{d-1}=dv_{d}$ and $a_d=(2v_{d}\mathcal C_d)^{-1}$ give
\[
\bigl\langle\widehat{|x|^{d-1}\sigma},\psi\bigr\rangle
=\frac{4a_d}{d}\int_{\mathbb R}\psi(t)\,dt.
\]

Thus the identity is proved for even test functions. Both distributions in the
statement of the lemma are even, so decomposing an arbitrary
$\psi\in C_c^\infty((-1,1))$ into its even and odd parts gives the desired
identity.
\end{proof}

\begin{lem}\label{lem-Phi12}
The functions $\Phi_1$, $\Phi_2$ defined in
\eqref{def-Phi12-odd} and \eqref{def-Phi12-even} are entire functions of order
at most~$1$ and finite exponential type.
\end{lem}

\begin{proof}
Suppose first that $d$ is odd. Group consecutive factors:
\[
\Phi_1(z)=\prod_{k=1}^\infty
\Bigl(1-\frac{z}{\tau_{2k-1}}\Bigr)
\Bigl(1+\frac{z}{\tau_{2k}}\Bigr).
\]
Since $1/\tau_k\downarrow0$,
\[
\sum_{k=1}^{\infty}
\Bigl(\frac1{\tau_{2k-1}}-\frac1{\tau_{2k}}\Bigr)<\infty,
\]
and $\tau_{2k-1}\tau_{2k}\ge ck^2$. Hence the product converges locally
uniformly and, for $r=|z|$,
\[
|\Phi_1(z)|\le e^{C_1r}
\prod_{k=1}^{\infty}\Bigl(1+\frac{C_2r^2}{k^2}\Bigr)
\le C_3e^{C_4r}.
\]
The same holds for $\Phi_2(z)=\Phi_1(-z)$.

Now suppose that $d$ is even. Local uniform convergence of the products
follows from $\sum_k\tau_k^{-2}<\infty$. Moreover,
$\tau_{2k-1},\tau_{2k}\ge ck$, and for $j=1,2$
\[
\log\max_{|z|\le r}|\Phi_j(z)|
\le\sum_{k=1}^{\infty}\log\Bigl(1+\frac{Cr^2}{k^2}\Bigr)=O(r).
\]
This proves the assertion in both cases.
\end{proof}

\section{Proof of Theorem \ref{thm-main}~\textup{(a)}}
\label{sec-functional}

We split the proof into short steps.

\subsection{The signature $\tilde\sigma$ and the measure $\nu$}
The signature $\tilde\sigma$ defined in \eqref{def-sigma-tilde} belongs to
$L^\infty(\mathbb R)$, is piecewise constant, and has parity
\begin{equation}\label{sigma-tilde-parity}
\tilde\sigma(-x)=(-1)^\epsilon\tilde\sigma(x).
\end{equation}
On the positive half-line, its jump at $\tau_k$ is $2(-1)^k$. At the origin,
the jump is~$2$ when $\epsilon=1$ and is absent when $\epsilon=0$. Therefore,
in the sense of locally finite signed measures,
\begin{equation}\label{atomic-nu}
\nu=2\epsilon\delta_0
+2\sum_{k=1}^{\infty}(-1)^k
\bigl(\delta_{\tau_k}+(-1)^d\delta_{-\tau_k}\bigr).
\end{equation}
Since $\tau_k\sim\pi k$, we have $\nu\in\mathcal S'(\mathbb R)$.

\subsection{The polynomial $P_d$}\label{subsec-Pd}
From \eqref{eq:extremality-fourier} and \eqref{weight-sigma-tilde}, we obtain
\begin{equation}\label{sigma-a}
\widehat{x^{d-1}\tilde\sigma}= i^{d-1}\partial_t^{d-1}\,\widehat{\tilde\sigma}
=\frac{4a_d}{d} \quad\text{on $(-1,1)$}.
\end{equation}
Hence the restriction of the distribution $\widehat{\tilde\sigma}$ to
$(-1,1)$ coincides with a polynomial of degree at most $d-1$. By \eqref{sigma-tilde-parity}, it is even when $\epsilon=0$ and odd
when $\epsilon=1$. Since $\tilde\sigma$ is real-valued, this polynomial is real-valued
when $\epsilon=0$ and purely imaginary when $\epsilon=1$. Therefore, after the
normalization
\[
P_d=2^{1-\epsilon}i^\epsilon\,\widehat{\tilde\sigma}\quad\text{on $(-1,1)$}
\]
we obtain the real polynomial defined in \eqref{intro-Pd},
\begin{equation}\label{def-Pd}
P_d(t)=\sum_{n=0}^Np_nt^{2n+\epsilon},\quad 2N+\epsilon=d-1.
\end{equation}
Now \eqref{sigma-a} implies
\[
2^{\epsilon-1}i^{d-1-\epsilon}
\partial_t^{d-1}P_d=\frac{4a_d}{d}
\]
or, since $d-1-\epsilon=2N$,
\[
(-1)^N2^{\epsilon-1}(d-1)!\,p_N=\frac{4a_d}{d}.
\]
Hence
\begin{equation}\label{Pd-top}
p_N=(-1)^N2^{1-\epsilon}\,\frac{4a_d}{d!}.
\end{equation}

\subsection{The function $M$}
In what follows, boundary values are understood in the sense of tempered distributions:
\[
F^\pm=\lim_{\varepsilon\downarrow0}F(\Cdot\pm i\varepsilon) \quad\text{in
$\mathcal S'(\mathbb R)$}.
\]
We also set
\[
t_+^r=\mathbf1_{(0,\infty)}(t)t^r,\quad t_-^r=\mathbf1_{(-\infty,0)}(t)t^r,
\quad r\in\mathbb Z_{\ge0}.
\]

Define the Cauchy transform of the measure $\nu$ by the symmetric principal
value
\[
M(z)=\pv\int_{\mathbb R}\frac{d\nu(t)}{z-t}
=\lim_{L\to\infty}\int_{-L}^L\frac{d\nu(t)}{z-t}.
\]
From \eqref{atomic-nu}, we obtain
\begin{equation}\label{def-M}
M(z)=\frac{2\epsilon}{z}
+4\sum_{k=1}^{\infty}(-1)^k\,\frac{z^\epsilon\tau_k^{1-\epsilon}}{z^2-\tau_k^2}.
\end{equation}
When $\epsilon=0$, a term of the series is
$-4(-1)^k/\tau_k+O_K(\tau_k^{-3})$ on each compact set $K$ away from the
poles, so local uniform convergence follows from the alternating series test.
When $\epsilon=1$, the series converges absolutely and locally uniformly away
from the poles. Hence $M$ is meromorphic, and
\begin{equation}\label{M-parity}
M(-z)=(-1)^\epsilon M(z).
\end{equation}
Its poles are simple and
\begin{equation}\label{M-residues}
\res_{z=0}M=2\epsilon,\quad
\res_{z=\tau_k}M=2(-1)^k,\quad
\res_{z=-\tau_k}M=2(-1)^{k+\epsilon+1}.
\end{equation}

For the function \eqref{intro-Psi},
\[
\Psi(z)=z^\epsilon\,\frac{\Phi_2(z)}{\Phi_1(z)}
\]
logarithmic differentiation of the products
\eqref{def-Phi12-odd}, \eqref{def-Phi12-even} gives
\begin{equation}\label{M-logder}
M(z)=2\,\frac{\Psi'(z)}{\Psi(z)} =\frac{2\epsilon}{z}
+2\,\frac{\Phi_2'(z)}{\Phi_2(z)} -2\,\frac{\Phi_1'(z)}{\Phi_1(z)}.
\end{equation}

\begin{lem}\label{lem-Cauchy}
The boundary values $M^\pm$ admit the representations
\begin{equation}\label{M-local-spectrum}
\begin{aligned}
\widehat{M^+}&=c_\epsilon t_+^1P_d(t)+T_+,
&& \supp T_+\subset[1,\infty),\\
\widehat{M^-}&=-c_\epsilon t_-^1P_d(t)+T_-,
&& \supp T_-\subset(-\infty,-1],
\end{aligned}
\end{equation}
where
\begin{equation}\label{def-cepsilon}
c_\epsilon=-2^\epsilon\pi i^{2-\epsilon}.
\end{equation}
\end{lem}

\begin{proof}
For $L>0$, let $\nu_L$ be the restriction of the measure $\nu$ to $[-L,L]$
and put
\[
M_L(z)=\int_{-L}^L\frac{d\nu(t)}{z-t}.
\]

We use the standard formulas (see \cite[Ch.~II]{Ge64})
\begin{equation}\label{boundary-Fourier}
\widehat{(x+i0)^{-1}}=-2\pi i\,\mathbf1_{(0,\infty)},\quad
\widehat{(x-i0)^{-1}}=2\pi i\,\mathbf1_{(-\infty,0)}.
\end{equation}
Hence, using the translation formulas,
\[
\widehat{M_L^+}
=-2\pi i\,\mathbf1_{(0,\infty)}\widehat{\nu_L},\quad
\widehat{M_L^-}
=2\pi i\,\mathbf1_{(-\infty,0)}\widehat{\nu_L}.
\]

We first show that $M_L^\pm$ converge in $\mathcal S'(\mathbb R)$. If
$\epsilon=0$, then the $k$th term of the series~\eqref{def-M} has the form
\[
2(-1)^k\Bigl(\frac1{z-\tau_k}-\frac1{z+\tau_k}\Bigr).
\]
As $\tau_k\to\infty$, its boundary value acts on a test function $\psi$ as
\[
-\frac{4(-1)^k}{\tau_k}\int_{\mathbb R}\psi(x)\,dx
+O_\psi(\tau_k^{-2}).
\]
The first series converges by the alternating series test, while the remainder
converges absolutely.

If $\epsilon=1$, then the $k$th term has the form
\[
2(-1)^k\Bigl(\frac1{z-\tau_k}+\frac1{z+\tau_k}\Bigr),
\]
and after combining the symmetric terms, its action on $\psi$ is
$O_\psi(\tau_k^{-2})$. Thus in this case the series converges absolutely.
Therefore $M_L^\pm$ converge in $\mathcal S'(\mathbb R)$ to $M^\pm$.

Passing to the limit, on $\mathbb R\setminus\{0\}$ we obtain
\begin{equation}\label{M-boundary-spectrum}
\widehat{M^+}
=-2\pi i\,\mathbf1_{(0,\infty)}\widehat\nu,\quad
\widehat{M^-}
=2\pi i\,\mathbf1_{(-\infty,0)}\widehat\nu.
\end{equation}
Moreover,
\[
\supp\widehat{M^+}\subset[0,\infty),\quad
\supp\widehat{M^-}\subset(-\infty,0].
\]

Since $\nu=\tilde\sigma'$, the definition of $P_d$ (see
Subsection~\ref{subsec-Pd}) gives
\[
\widehat\nu=it\,\widehat{\tilde\sigma} =2^{\epsilon-1}i^{1-\epsilon}tP_d\quad
\text{on $(-1,1)$}.
\]
Hence, by \eqref{M-boundary-spectrum},
\[
\widehat{M^+}=c_\epsilon t_+^1P_d(t)
\quad\text{on }(-1,1)\setminus\{0\}.
\]
Therefore
\[
\widehat{M^+}=c_\epsilon t_+^1P_d(t)+S_0+T_+,
\]
where
\[
\supp S_0\subset\{0\},\quad \supp T_+\subset[1,\infty).
\]

We show that $S_0=0$. For $\epsilon=0$, \eqref{def-M} gives
\[
M(iy)=-4\sum_{k=1}^\infty(-1)^k
\frac{\tau_k}{y^2+\tau_k^2}.
\]
The sequence $\tau_k/(y^2+\tau_k^2)$ is unimodal, and its maximum does not
exceed $1/(2y)$. Splitting the alternating sum near the maximum and applying
the alternating series test to both parts, we obtain $M(iy)=O(y^{-1})$. For
$\epsilon=1$, the same test applied to \eqref{def-M} gives
\[
|M(iy)|\le\frac2y+\frac{4y}{y^2+\tau_1^2}=O(y^{-1}).
\]
Thus
\begin{equation}\label{M-iy-decay}
M(iy)=O(y^{-1}),\quad y\to+\infty.
\end{equation}

The distribution $S_0$ is a finite linear combination of
$\delta_0,\delta_0',\ldots.$ If $S_0\ne0$, its contribution to the Laplace
representation of $M(iy)$ is a nonzero polynomial in $y$. The contribution
of $t_+^1P_d$ decays polynomially, while the contribution of $T_+$ decays
exponentially up to a polynomial factor. This contradicts
\eqref{M-iy-decay}. Hence $S_0=0$, and the first representation in
\eqref{M-local-spectrum} is proved.

The second representation follows from \eqref{M-parity} and the identity
$P_d(-t)=(-1)^\epsilon P_d(t)$.
\end{proof}

\subsection{The spectral gap for $\Theta_d$}
We use the formula
\[
\widehat{x^r}=2\pi i^r\delta_0^{(r)},
\quad r\in\mathbb Z_{\ge0},
\]
and the following lemma.

\begin{lem}[\cite{Go26}]\label{lem-truncated-power}
Let $r,L\in\mathbb Z_{\ge0}$. Then in $\mathcal S'(\mathbb R)$
\[
\partial_t^Lt_+^r=
\begin{cases}
\dfrac{r!}{(r-L)!}\,t_+^{r-L},&0\le L\le r,\\[3mm]
r!\,\delta_0^{(L-r-1)},&L\ge r+1.
\end{cases}
\]
\end{lem}

Recall the definition of the polynomial $R_d$ given in \eqref{intro-Rd}:
\[
R_d(z)=\frac\epsilon2\,z^{2(N+\epsilon)} +2^{\epsilon-3}\sum_{n=0}^N(-1)^n
(2n+\epsilon+1)!\,p_nz^{2(N-n)} =\sum_{j=0}^{N+\epsilon}r_jz^{2j}.
\]
By \eqref{Pd-top},
\begin{equation}\label{Rd-zero}
R_d(0)=r_0=2^{\epsilon-3}(-1)^N d!\,p_N=a_d.
\end{equation}
Moreover, $R_d$ is even and
\[
\deg R_d\le d-1+\epsilon.
\]

Consider the function
\begin{equation}\label{def-Theta}
\Theta_d(z)=z^dM(z)-2\epsilon z^{d-1}+\frac{4R_d(z)}z.
\end{equation}
Its boundary values are
\[
\Theta_d^\pm(x)=x^dM^\pm(x)-2\epsilon x^{d-1}
+\frac{4R_d(x)}{x\pm i0}.
\]

We prove
\begin{equation}\label{Theta-gap}
\supp\widehat{\Theta_d^+}\subset[1,\infty),\quad
\supp\widehat{\Theta_d^-}\subset(-\infty,-1].
\end{equation}
By Lemma~\ref{lem-Cauchy},
\begin{equation}\label{x-d-M-spectrum}
\widehat{x^dM^+}
=c_\epsilon i^d\sum_{n=0}^Np_n
\partial_t^dt_+^{2n+\epsilon+1}+T_1,
\quad \supp T_1\subset[1,\infty).
\end{equation}
For $n=N$, we have $2N+\epsilon+1=d$. From \eqref{Pd-top},
\eqref{def-cepsilon}, and Lemma~\ref{lem-truncated-power}, we obtain the term
$8\pi ia_d\,\mathbf1_{(0,\infty)}$. It cancels with the Fourier transform of
$4a_d/(x+i0)$.

Let $0\le n<N$ and $r=d-2n-\epsilon-2$. The corresponding term in
\eqref{x-d-M-spectrum} is
\[
c_\epsilon i^d(2n+\epsilon+1)!\,p_n\delta_0^{(r)}.
\]
The monomial of degree $r$ in
\[
\frac{4R_d(x)}{x+i0}-2\epsilon x^{d-1}
\]
has coefficient $2^{\epsilon-1}(-1)^n(2n+\epsilon+1)!\,p_n$. Since
$c_\epsilon i^d=-2^\epsilon\pi(-1)^ni^r$, the Fourier transform of this
monomial is the negative of the preceding term. Thus all low-frequency terms cancel, and the
first inclusion in \eqref{Theta-gap} is proved. The second follows by complex
conjugation and reflection $t\mapsto-t$.

\subsection{The polynomial defect $H_d$}
We shall need the following lemma.

\begin{lem}[\cite{Go26}]\label{lem-multiplier}
Let $f\in\PW_1^\infty(\mathbb R)$ and $T\in\mathcal S'(\mathbb R)$. Then
$fT\in\mathcal S'(\mathbb R)$ and
\[
\widehat{fT}=\frac1{2\pi}\,\widehat f*\widehat T, \quad
\supp\widehat{fT}\subset\supp\widehat f+\supp\widehat T.
\]
\end{lem}

Introduce
\begin{equation}\label{def-Hd}
H_d(z)=\varphi(z)\Theta_d(z)-\frac{4a_d}{z}.
\end{equation}
The poles of $\Theta_d$ at $\pm\tau_k$ are canceled by the simple zeros of
$\varphi$. At the origin, \eqref{def-M} and \eqref{Rd-zero} give
\[
\Theta_d(z)=\frac{4a_d}{z}+O(z),\quad
\varphi(z)=1+O(z^2),
\]
so the singularity there is also removable. Hence $H_d$ is an entire odd
function.

Since $|x|^{d-1}\varphi(x)\in L^1(\mathbb R)$ and
$\|\varphi\|_{\infty}=1$, we have $\varphi\in\PW_1^1(\mathbb R)$ and
$\supp\widehat\varphi\subset[-1,1]$. Lemma~\ref{lem-multiplier},
\eqref{Theta-gap}, and \eqref{boundary-Fourier} imply
\[
\supp\widehat{H_d^+}\subset[0,\infty),\quad
\supp\widehat{H_d^-}\subset(-\infty,0].
\]
Both boundary values are restrictions of the same entire function, hence
\[
\supp\widehat H_d\subset\{0\}.
\]
Therefore $H_d$ is a polynomial. By oddness,
\[
H_d(z)=z\Pi_d(z)
\]
with an even polynomial $\Pi_d$.

\subsection{Vanishing of the defect}
We will use the Plancherel--P\'olya inequality in the following form.

\begin{lem}[\cite{Go26}]\label{lem-PP}
If $g\in\PW_1^1(\mathbb R)$ and the sequence
$(x_n)_{n\ge1}\subset\mathbb R$ is uniformly separated, then
\[
\sum_{n=1}^{\infty}|g'(x_n)|\le C\|g\|_1.
\]
\end{lem}

Letting $z=\tau_k$ in \eqref{def-Hd} and using
\eqref{M-residues}, we obtain
\begin{equation}\label{Pi-at-zero}
\Pi_d(\tau_k)=2(-1)^k\tau_k^{d-1}\varphi'(\tau_k)
-\frac{4a_d}{\tau_k^2}.
\end{equation}
The function $g(z)=z^{d-1}\varphi(z)$ belongs to
$\PW_1^1(\mathbb R)$ and
$g'(\tau_k)=\tau_k^{d-1}\varphi'(\tau_k)$. From the sequence $(\tau_k)$ one
can choose a uniformly separated subsequence $(\tau_{k_j})_{j\ge1}$. By
Lemma~\ref{lem-PP},
\[
\sum_{j\ge 1}|g'(\tau_{k_j})|<\infty.
\]
If $\Pi_d\not\equiv0$, then
$|\Pi_d(x)+4a_d/x^2|$ is bounded away from zero for all sufficiently large~$x$. It would then follow from \eqref{Pi-at-zero} that
$|g'(\tau_{k_j})|$ is bounded away from zero, which is impossible. Hence
$\Pi_d\equiv0$, and
\begin{equation}\label{phi-Theta}
\varphi(z)\Theta_d(z)=\frac{4a_d}{z}.
\end{equation}

\subsection{The final functional equation}
From \eqref{M-logder}, we have
\[
\varphi M=\frac{2\epsilon\varphi}{z}
-2\bigl(\Phi_1'\Phi_2-\Phi_1\Phi_2'\bigr).
\]
Substituting this identity and \eqref{def-Theta} into \eqref{phi-Theta},
multiplying by $z$, and canceling the terms
$2\epsilon z^d\varphi$, we obtain
\[
-2z^{d+1}\bigl(\Phi_1'\Phi_2-\Phi_1\Phi_2'\bigr)
+4R_d\Phi_1\Phi_2=4a_d.
\]
Dividing by $-2$ gives \eqref{intro-FE}.

\section{Proof of Theorem \ref{thm-main}~\textup{(b)}}
\label{sec-ode}

\subsection{Relations at the zeros of $\Phi_1$ and $\Phi_2$}
Differentiate \eqref{intro-FE}. If $s$ is a zero of $\Phi_1$, then
$\Phi_2(s)\ne0$, and after cancellation we obtain
\begin{equation}\label{zero-relation-1}
s^{d+1}\Phi_1''(s)
+\bigl((d+1)s^d-2R_d(s)\bigr)\Phi_1'(s)=0.
\end{equation}
If $s$ is a zero of $\Phi_2$, then similarly
\begin{equation}\label{zero-relation-2}
s^{d+1}\Phi_2''(s)
+\bigl((d+1)s^d+2R_d(s)\bigr)\Phi_2'(s)=0.
\end{equation}

We use the following lemma on the logarithmic derivative.

\begin{lem}[\cite{Go26}]\label{lem-logder}
Let $f$ be an entire function of finite order and let $U_1$, $U_2$ be
polynomials.~If
\[
G=U_2\,\frac{f''}{f}+U_1\frac{f'}{f}
\]
has no poles, then it is a polynomial.
\end{lem}

\subsection{The polynomials $\mathcal P_1$, $\mathcal P_2$, and $\mathcal K_d$}
Consider the expressions \eqref{def-P1}, \eqref{def-P2}:
\begin{align*}
\mathcal P_1(z)&=
-\frac{z^{d+1}\Phi_1''+\bigl((d+1)z^d-2R_d\bigr)\Phi_1'}{\Phi_1}, \\
\mathcal P_2(z)&=
-\frac{z^{d+1}\Phi_2''+\bigl((d+1)z^d+2R_d\bigr)\Phi_2'}{\Phi_2}.
\end{align*}
By \eqref{zero-relation-1}, \eqref{zero-relation-2}, all possible poles are
removable. Lemma~\ref{lem-logder} shows that $\mathcal P_1$ and
$\mathcal P_2$ are polynomials with real coefficients.

Differentiating \eqref{intro-FE} and substituting the definitions
\eqref{def-P1}, \eqref{def-P2}, we obtain
\[
(\mathcal P_1-\mathcal P_2+2R_d')\Phi_1\Phi_2=0.
\]
Hence $\mathcal P_2=\mathcal P_1+2R_d'$. For the moment, introduce the
polynomial
\begin{equation}\label{def-calKd}
\mathcal K_d=\frac{\mathcal P_1+\mathcal P_2}{2}.
\end{equation}
We will show below that $\mathcal K_d$ coincides with the polynomial $K_d$
defined in \eqref{intro-Kd}. Since
\[
\mathcal P_1=\mathcal K_d-R_d',\quad
\mathcal P_2=\mathcal K_d+R_d',
\]
the functions $\Phi_1$, $\Phi_2$ satisfy the system
\begin{equation}\label{calK-ODE}
\begin{cases}
z^{d+1}\Phi_1''+\bigl((d+1)z^d-2R_d\bigr)\Phi_1'
+\bigl(\mathcal K_d-R_d'\bigr)\Phi_1=0,\\
z^{d+1}\Phi_2''+\bigl((d+1)z^d+2R_d\bigr)\Phi_2'
+\bigl(\mathcal K_d+R_d'\bigr)\Phi_2=0.
\end{cases}
\end{equation}

\subsection{Differential equation for $\tilde\Phi_j$}
Direct substitution into \eqref{calK-ODE} and the identity
$h_d'=R_d/z^{d+1}$ show that the functions $\tilde\Phi_j$ satisfy
\begin{equation}\label{calK-common-ODE}
y''+\frac{d+1}{z}\,y'
+\Bigl(\frac{\mathcal K_d(z)}{z^{d+1}}
-\frac{R_d(z)^2}{z^{2d+2}}\Bigr)y=0.
\end{equation}
It follows from \eqref{intro-FE} that
\begin{equation}\label{tilde-Wronskian}
W(\tilde\Phi_1,\tilde\Phi_2)
=\frac{2a_d}{z^{d+1}}.
\end{equation}
In particular, the functions $\tilde\Phi_j$ form a fundamental system of
solutions of \eqref{calK-common-ODE}.

From the explicit formula \eqref{intro-hd}, on the positive half-line we have
\begin{equation}\label{hd-infinity}
h_d(x)=\frac\epsilon2\log x+O(x^{-1-\epsilon}),\quad
h_d'(x)=\frac\epsilon{2x}+O(x^{-2-\epsilon}),
\quad x\to+\infty.
\end{equation}

\subsection{Liouville transformation and zero density}
We shall need the following two statements.

\begin{lem}[\cite{Go26}]\label{lem-Liouville}
Let $p\in C^1[A,B]$, $q\in C[A,B]$, and let $y\in C^2[A,B]$ satisfy
\[
y''+py'+qy=0. 
\]
Then
\[
u(x)=\exp\Bigl(\frac12\int_A^xp\Bigr)y(x)
\]
satisfies
\[
u''+Vu=0,\quad V=q-\frac{p'}2-\frac{p^2}{4}.
\]
\end{lem}

\begin{lem}[\cite{Go26}]\label{lem-zero-count}
Let $V\in C[A,B]$ be real-valued, let $u\in C^2[A,B]\setminus\{0\}$, and
assume that 
\[
u''+Vu=0\quad \text{on $[A,B]$}.
\]

\textup{(a)} If $V\le\Omega^2$ on $[A,B]$, where $\Omega\ge0$, then the number of
distinct zeros of $u$ on $[A,B]$ does not exceed
\[
\frac{\Omega}{\pi}\,(B-A)+1.
\]

\textup{(b)} If $V\ge\omega^2$ on $[A,B]$, where $\omega\ge0$, then the number of
distinct zeros of $u$ on $[A,B]$ is at least
\[
\frac{\omega}{\pi}\,(B-A)-1.
\]
\end{lem}

Apply Lemma~\ref{lem-Liouville} to \eqref{calK-common-ODE} on the positive
half-line and put
\begin{equation}\label{def-u}
u_j(x)=x^{(d+1)/2}\tilde\Phi_j(x),\quad j=1,2.
\end{equation}
Then
\begin{equation}\label{Schrodinger}
u_j''+\mathcal V_d(x)u_j=0,
\end{equation}
where
\[
\mathcal V_d(x)=\frac{\mathcal K_d(x)}{x^{d+1}}
-\frac{R_d(x)^2}{x^{2d+2}}
-\frac{d^2-1}{4x^2}.
\]
Since $\deg R_d\le d-1+\epsilon\le d$, we have
\begin{equation}\label{calV-asymp-K}
\mathcal V_d(x)=\frac{\mathcal K_d(x)}{x^{d+1}}+O(x^{-2}),
\quad x\to+\infty.
\end{equation}

By \eqref{def-Phi12-tilde} and \eqref{def-u},
\[
u_1(x)=x^{(d+1)/2}e^{-h_d(x)}\Phi_1(x),
\]
so the positive zeros of $u_1$ coincide with the positive zeros of $\Phi_1$,
namely $\tau_1,\tau_3,\tau_5,\ldots.$ If $n_{u_1}(X)$ denotes the number of
these zeros in $(0,X]$, then $\tau_k\sim\pi k$ gives
\begin{equation}\label{u-zero-density}
n_{u_1}(X)=\frac{X}{2\pi}+o(X),\quad X\to\infty.
\end{equation}

\subsection{Degree of the polynomial $\mathcal K_d$}
If $\mathcal K_d\equiv0$ or $\deg\mathcal K_d\le d$, then by
\eqref{calV-asymp-K},
\[
\mathcal V_d(x)\to0.
\]
For every $\delta>0$ and all sufficiently large $x$, we have
$\mathcal V_d(x)\le\delta$. Lemma~\ref{lem-zero-count}~\textup{(a)} gives
\[
\limsup_{X\to\infty}\frac{n_{u_1}(X)}X
\le\frac{\sqrt\delta}{\pi},
\]
which after $\delta\downarrow0$ contradicts \eqref{u-zero-density}. Hence
\[
\deg\mathcal K_d\ge d+1.
\]

Let $s=\deg\mathcal K_d$, let $c\ne0$ be the leading coefficient of
$\mathcal K_d$, and let $\kappa=s-d-1\ge0$. Then
\[
\mathcal V_d(x)=cx^\kappa(1+o(1)).
\]
If $c<0$, then $\mathcal V_d(x)<0$ for all sufficiently large $x$.
Lemma~\ref{lem-zero-count}~\textup{(a)}, applied with $\Omega=0$, then
implies that $u_1$ has at most one zero on a sufficiently large half-line,
contradicting~\eqref{u-zero-density}. Therefore $c>0$.

If $\kappa>0$, then on $[X,2X]$ for large $X$,
\[
\mathcal V_d(x)\ge\frac{c}{2}\,X^\kappa.
\]
By Lemma~\ref{lem-zero-count}~\textup{(b)},
\[
n_{u_1}(2X)-n_{u_1}(X) \ge\frac1\pi\,\sqrt{\frac{c}{2}}\,X^{1+\kappa/2}-2,
\]
which contradicts the linear asymptotic relation
\eqref{u-zero-density}. Hence
\begin{equation}\label{degree-exact}
\deg\mathcal K_d=d+1.
\end{equation}

\subsection{Leading coefficient of $\mathcal K_d$}
It remains to determine the leading coefficient $c$. By
\eqref{calV-asymp-K}, \eqref{degree-exact},
\[
\mathcal V_d(x)=c+o(1).
\]
For every $0<\delta<c$ and all sufficiently large $x$,
\[
c-\delta\le\mathcal V_d(x)\le c+\delta.
\]
Lemma~\ref{lem-zero-count} and \eqref{u-zero-density}, followed by
$\delta\downarrow0$, give
\[
\frac{\sqrt c}{\pi}=\frac1{2\pi}.
\]
Hence
\begin{equation}\label{calKd-leading}
c=\frac14.
\end{equation}

\subsection{The identity $\mathcal K_d=K_d$}
Put
\[
L_j(z)=\frac{\Phi_j'(z)}{\Phi_j(z)},\quad j=1,2.
\]
From \eqref{def-P1}, \eqref{def-P2}, \eqref{def-calKd}, we obtain
\begin{equation}\label{calK-via-L12}
\mathcal K_d= -\frac{z^{d+1}}2\,\bigl(L_1'+L_1^2+L_2'+L_2^2\bigr)
-\frac{d+1}{2}\,z^d(L_1+L_2) +R_d(L_1-L_2).
\end{equation}
Since
\[
L_1+L_2=\frac{\varphi'}{\varphi}=O(z),\quad z\to0,
\]
the first two terms on the right-hand side of \eqref{calK-via-L12} are
$O(z^{d+1})$. Hence
\begin{equation}\label{calK-low}
\mathcal K_d(z)=R_d(z)\bigl(L_1(z)-L_2(z)\bigr)+O(z^{d+1}).
\end{equation}

On the other hand, by \eqref{M-logder},
\[
L_1-L_2=\frac{\epsilon}{z}-\frac{M}{2}.
\]
Near the origin,
\[
\frac1{z^2-\tau_k^2}
=-\sum_{q=0}^\infty\frac{z^{2q}}{\tau_k^{2q+2}},
\]
so \eqref{def-M} gives
\[
M(z)=\frac{2\epsilon}{z}
-4\sum_{q=0}^\infty\lambda_qz^{2q+\epsilon},
\]
where, recall (see \eqref{def-lambda}),
\[
\lambda_q=\sum_{k=1}^\infty
\frac{(-1)^k}{\tau_k^{2q+\epsilon+1}}.
\]
Therefore
\begin{equation}\label{L1-L2-lambda}
L_1(z)-L_2(z)
=2\sum_{q=0}^{\infty}\lambda_qz^{2q+\epsilon}.
\end{equation}
Since
\[
R_d(z)=\sum_{j=0}^{N+\epsilon}r_jz^{2j},
\]
it follows from \eqref{calK-low}, \eqref{L1-L2-lambda} that for
$0\le n\le N$ the coefficient of $z^{2n+\epsilon}$ in the polynomial~$\mathcal K_d$~is
\begin{equation}\label{calK-low-coeff}
2\sum_{j=0}^nr_j\lambda_{n-j}.
\end{equation}

From \eqref{calK-low-coeff}, \eqref{degree-exact}, and
\eqref{calKd-leading}, we obtain
\[
\mathcal K_d(z)=\frac14\,z^{d+1} +2\sum_{n=0}^N
\biggl(\sum_{j=0}^nr_j\lambda_{n-j}\biggr)z^{2n+\epsilon}.
\]
By definition \eqref{intro-Kd}, the right-hand side is $K_d(z)$. Hence
$\mathcal K_d=K_d$. Therefore the system \eqref{calK-ODE} coincides with
\eqref{intro-ODE}, equation \eqref{calK-common-ODE} coincides with
\eqref{intro-common-ODE}, and $\mathcal V_d=V_d$. Together with
\eqref{tilde-Wronskian}, this completes the proof of part~\textup{(b)}.

\section{Proof of the corollaries}\label{sec-corollaries}

\subsection{Proof of Corollary~\ref{cor-liouville}}
The corollary has essentially already been proved above in
Section~\ref{sec-ode}. After establishing the identities
$\mathcal K_d=K_d$, $\mathcal V_d=V_d$, we obtain equation~\eqref{Schrodinger} with the potential \eqref{intro-V}. Moreover, from
\eqref{tilde-Wronskian},
\[
W(u_1,u_2) =x^{d+1}W(\tilde\Phi_1,\tilde\Phi_2) =2a_d.
\]
In particular, $u_1$ and $u_2$ are linearly independent.

\subsection{Proof of Corollary~\ref{cor-Phi-growth}}
From the explicit formulas for $K_d$ and $R_d$, we have
\[
K_d(x)=\frac14\,x^{d+1}+O(x^{d-1}),\quad
R_d(x)=O(x^{d-1+\epsilon}).
\]
By \eqref{intro-V},
\begin{equation}\label{V-quarter}
V_d(x)=\frac14+O(x^{-2}),\quad x\to+\infty.
\end{equation}
For a real solution of \eqref{intro-Liouville-ODE}, introduce the energy
function
\[
\mathcal E(x)=u'(x)^2+\frac14\,u(x)^2.
\]
Then
\[
\mathcal E'(x)
=-2\Bigl(V_d(x)-\frac14\Bigr)u(x)u'(x),
\]
and by \eqref{V-quarter},
\[
|\mathcal E'(x)|\le Cx^{-2}\mathcal E(x)
\]
for all sufficiently large $x$. Hence $\mathcal E(x)=O(1)$ and
\[
u_j(x)=O(1),\quad x\to+\infty,\quad j=1,2.
\]
By the definition of $u_j$,
\[
\tilde\Phi_j(x)=O\bigl(x^{-(d+1)/2}\bigr),\quad j=1,2.
\]
Using \eqref{def-Phi12-tilde}, \eqref{hd-infinity}, we obtain
\[
\Phi_1(x)=O\bigl(x^{-(d+1-\epsilon)/2}\bigr),\quad
\Phi_2(x)=O\bigl(x^{-(d+1+\epsilon)/2}\bigr),
\quad x\to+\infty.
\]
For odd $d$, we have $\Phi_1(-x)=\Phi_2(x)$, while for even $d$ both functions
$\Phi_j$ are even. Hence the same estimates hold as $x\to-\infty$, and
\eqref{intro-decay} is proved.

By Lemma~\ref{lem-Phi12}, the functions $\Phi_j$ have finite exponential type,
while \eqref{intro-decay} shows that their restrictions to $\mathbb R$ are
tempered distributions. Therefore, by the Paley--Wiener--Schwartz theorem,
the Fourier transforms $\widehat{\Phi_j}$ have compact support.

Apply the Fourier transform to equations \eqref{intro-ODE}. Since the leading
term of $K_d$ is $z^{d+1}/4$, the terms
$z^{d+1}\Phi_j''$ and $K_d\Phi_j$ give the coefficient, up to a nonzero
constant factor,
$1/4-t^2$
of the highest derivative $\partial_t^{d+1}\widehat{\Phi_j}$. Therefore, on the intervals
$(-\infty,-1/2)$ and $(1/2,\infty)$, the distribution
$\widehat{\Phi_j}$ is a smooth solution of a regular linear differential
equation. Since it has compact support, uniqueness gives
\[
\supp\widehat{\Phi_j}\subset[-1/2,1/2],\quad j=1,2,
\]
and the exponential type of each $\Phi_j$ is at most $1/2$.

Since $\varphi=\Phi_1\Phi_2$ has exact exponential type~$1$, while the type of
a product does not exceed the sum of the types of its factors, both
functions $\Phi_j$ have exact exponential type $1/2$.

\subsection{Proof of Corollary~\ref{cor-phi-ODE}}
By Theorem~\ref{thm-main}~(b), the functions $\tilde\Phi_1$,
$\tilde\Phi_2$ satisfy, for $z\ne0$, the equation
\[
y''+\frac{d+1}{z}\,y'+q_dy=0,
\]
and $\tilde\Phi_1\tilde\Phi_2=\varphi$. It is easy to check that if
$y_1$, $y_2$ satisfy
\[
y''+py'+qy=0,
\]
then their product $u=y_1y_2$ satisfies
\[
u'''+3pu''+(p'+4q+2p^2)u'
+(2q'+4pq)u=0.
\]
It remains to put $p=(d+1)/z$ and $q=q_d$.

\section{The case $d=2$}\label{sec-d2}

As a new example, consider $d=2$. Then $\epsilon=1$, $N=0$,
\[
v_{2}=\pi,\quad a_2=\frac1{2\pi\mathcal C_2},\quad P_2(t)=2a_2t,
\]
\[
R_2(z)=\frac12\,z^2+a_2,\quad
\lambda_0=\sum_{k=1}^\infty\frac{(-1)^k}{\tau_k^2}<0,
\]
\[
K_2(z)=\frac14\,z^3+\beta z,\quad
\beta=2a_2\lambda_0<0.
\]

For the numerical determination of the parameters $a_2$ and $\beta$, it is
convenient to use the system~\eqref{intro-ODE}, which in this case becomes
\begin{equation}\label{d2-ODE}
\begin{cases}
z^3\Phi_1''+(2z^2-2a_2)\Phi_1'
+\Bigl(\dfrac14\,z^3+(\beta-1)z\Bigr)\Phi_1=0,\\[2mm]
z^3\Phi_2''+(4z^2+2a_2)\Phi_2'
+\Bigl(\dfrac14\,z^3+(\beta+1)z\Bigr)\Phi_2=0.
\end{cases}
\end{equation}

Since $\Phi_j$ are even, write
\[
\Phi_j(z)=\sum_{n=0}^\infty c_{j,n}z^{2n},\quad
c_{j,0}=1,\quad j=1,2.
\]
Substitution into \eqref{d2-ODE} gives the three-term recurrences
\begin{equation}\label{d2-rec}
\begin{cases}
4a_2(n+1)c_{1,n+1}
=\bigl(2n(2n+1)+\beta-1\bigr)c_{1,n}
+\dfrac14\,c_{1,n-1},\\[2mm]
4a_2(n+1)c_{2,n+1}
=-\bigl(2n(2n+3)+\beta+1\bigr)c_{2,n}
-\dfrac14\,c_{2,n-1},
\end{cases}
\end{equation}
where $c_{1,-1}=c_{2,-1}=0$.

For fixed $a_2$ and $\beta$, each recurrence in \eqref{d2-rec} has two
asymptotically different solutions. For the dominant solutions,
\[
\frac{c_{1,n+1}}{c_{1,n}}\sim\frac{n}{a_2},\quad
\frac{c_{2,n+1}}{c_{2,n}}\sim-\frac{n}{a_2},
\]
so the corresponding power series do not define entire functions. For the
other solutions,
\begin{equation}\label{d2-minimal-asymp}
\frac{c_{j,n}}{c_{j,n-1}}\sim-\frac1{16n^2},
\quad j=1,2.
\end{equation}
These are minimal solutions: the ratio of a minimal solution to any linearly
independent solution of the same recurrence tends to zero. Since the
functions~$\Phi_j$ are entire, their coefficient sequences correspond
precisely to the minimal solutions.

The asymptotic relation \eqref{d2-minimal-asymp} also gives an independent
check of the exponential type of the functions $\Phi_j$. It implies
\[
\frac{2}{e}\lim_{n\to\infty}n|c_{j,n}|^{1/(2n)}=\frac12,\quad j=1,2,
\]
which agrees with the type of $\Phi_j$. This also confirms
Corollary~\ref{cor-Phi-growth} for $d=2$.

To compute the minimal solutions, we use Miller's algorithm
(see \cite{Ga67}). Define the residuals
$F_{j,J}(a_2,\beta)$, $j=1,2$. For fixed $a_2$, $\beta$, and $J$, they are
computed as follows: set
\[
c_{j,J}^{(J)}=1,\quad c_{j,J+1}^{(J)}=0,
\]
run the corresponding recurrence \eqref{d2-rec} backward down to the index
$-1$, and put
\[
F_{j,J}(a_2,\beta) =\frac{c_{j,-1}^{(J)}}{c_{j,0}^{(J)}}.
\]
Here the denominator does not vanish when $c_{j,-1}^{(J)}=0$, since
$c_{j,-1}^{(J)}=c_{j,0}^{(J)}=0$ would imply, by the recurrence,
$c_{j,n}^{(J)}=0$ for all $n$, contradicting $c_{j,J}^{(J)}=1$.

For a chosen $J$, we find the parameters $a_2$ and $\beta$ by Newton's method
from the system
\begin{equation}\label{d2-numerical-system}
F_{1,J}(a_2,\beta)=0,\quad
F_{2,J}(a_2,\beta)=0.
\end{equation}
For the initial value of $J$, we may start with
\[
a_2=7,\quad \beta=-0.3.
\]
Then we increase $J$ and, for each new value, solve
\eqref{d2-numerical-system}, using the pair $(a_2,\beta)$ found at the
preceding step as the initial approximation for Newton's method. The
resulting values stabilize rapidly as $J$ increases. After determining the
parameters, the minimal solutions are normalized by $c_{j,0}=1$.

A numerical computation at $120$-digit precision gives
\[
\begin{array}{c|r@{}l}
\text{quantity}&
\multicolumn{2}{c}{\text{numerical value}}\\ \hline
a_2& 7&.08855214513069412139\ldots\\
\beta& -0&.30466567351975548186\ldots\\
\mathcal C_2=1/(2\pi a_2)& 0&.02245239081738615612\ldots\\
\mathcal L^*(2)=2/a_2& 0&.28214506418970912412\ldots
\end{array}
\]

\section{Zeta interpretation}\label{sec-zeta}

In \cite{Bo25}, for $d=1$, a relation was established between the constant
$a_1=R_1$ and a zeta function built from the zeros $\tau_k$ of the extremal
function. In \cite{Go26}, this relation was extended to all odd $d$. Here we
consider an arbitrary dimension. In particular, we show that the coefficients
$r_j$ of the polynomial
\[
R_d(z)=\sum_{j=0}^{N+\epsilon}r_jz^{2j}
\]
have a natural interpretation as regularized values of alternating power sums
over the zeros of the extremal function.

\subsection{Mellin formula}
For $\Re s>0$, put
\begin{equation}\label{eta-def}
\eta_{\tau,d}(s)=
\sum_{k=1}^{\infty}\frac{(-1)^{k-1}}{\tau_k^s}.
\end{equation}
Since $\tau_k\sim\pi k$, the series converges conditionally for
$0<\Re s\le1$ and absolutely for $\Re s>1$.

For the function $M$ in \eqref{def-M}, put
\[
\tilde M(z)=M(z)-\frac{2\epsilon}{z}, \quad
\chi_0(s)=\cos\frac{\pi s}{2}, \quad
\chi_1(s)=\sin\frac{\pi s}{2}.
\]
We have
\[
\tilde M(iy)=4i^\epsilon\sum_{k=1}^{\infty}(-1)^{k-1}\,
\frac{y^\epsilon\tau_k^{1-\epsilon}}{y^2+\tau_k^2}.
\]
Since
\[
\int_0^\infty \frac{y^\epsilon\tau^{1-\epsilon}}{y^2+\tau^2}\,y^{-s}\,dy
=\frac{\pi\tau^{-s}} {2\sin\bigl(\pi(\epsilon-s+1)/2\bigr)}
=\frac{\pi\tau^{-s}}{2\chi_\epsilon(s)},
\]
in the strip $0<\Re s<1$ we obtain
\begin{equation}\label{eta-Mellin}
\eta_{\tau,d}(s)= \frac{\chi_\epsilon(s)}{2\pi i^\epsilon} \int_0^\infty \tilde
M(iy)y^{-s}\,dy.
\end{equation}
Termwise integration is justified by summation by parts, while convergence at
infinity follows from \eqref{M-iy-decay}.

\subsection{Analytic continuation}
Since $\supp\widehat{M^+}\subset[0,\infty)$, for $y>0$ the Fourier--Laplace
inversion formula gives
\[
M(iy)=\frac1{2\pi}\, \bigl\langle \widehat{M^+}(t),e^{-yt}\bigr\rangle.
\]
Hence, by Lemma~\ref{lem-Cauchy} and the identity
\[
\int_0^\infty t^m e^{-yt}\,dt=\frac{m!}{y^{m+1}}
\]
we obtain, as $y\to+\infty$, 
\begin{equation}\label{Mtilde-asympt}
\tilde M(iy) =-4\sum_{j=0}^{N+\epsilon} \frac{r_j}{(iy)^{d-2j+1}} +O(y^Le^{-y})
\end{equation}
with some integer $L\ge0$.

Formula \eqref{Mtilde-asympt} allows us to analytically continue
\eqref{eta-Mellin} to the whole complex plane. For $\Re s<1$, put
\[
\begin{aligned}
I_0(s)&=\int_0^1\tilde M(iy)y^{-s}\,dy,\\ I_\infty(s)&=\int_1^\infty
\biggl(\tilde M(iy)+4\sum_{j=0}^{N+\epsilon}
\frac{r_j}{(iy)^{d-2j+1}}\biggr)y^{-s}\,dy.
\end{aligned}
\]
Then, by analytic continuation,
\begin{equation}\label{eta-cont}
\eta_{\tau,d}(s) =\frac{\chi_\epsilon(s)}{2\pi i^\epsilon}
\biggl(I_0(s)+I_\infty(s) -4\sum_{j=0}^{N+\epsilon} \frac{r_j
i^{-(d-2j+1)}}{s+d-2j}\biggr).
\end{equation}
The function $I_0$ is holomorphic for $\Re s<1$, while $I_\infty$ is entire.
The possible poles of the last sum are canceled by the corresponding simple
zeros of $\chi_\epsilon$. Hence $\eta_{\tau,d}$ extends to an entire function.

\subsection{Zeta interpretation of the coefficients of $R_d$}
From \eqref{eta-cont}, for all $0\le j\le N+\epsilon$, we obtain
\begin{equation}\label{eta-values}
\eta_{\tau,d}(-d+2j)=r_j.
\end{equation}
In particular,
\[
\eta_{\tau,d}(-d)=r_0=a_d.
\]

\subsection{Value at zero}
For even $d$, the endpoint case $j=N+1$ in \eqref{eta-values} immediately
gives
\[
\eta_{\tau,d}(0)=r_{N+1}=\frac12.
\]

We show that the same identity holds for odd $d$. In this case
$\epsilon=0$, $\tilde\sigma=\sigma$, and by~\eqref{def-Pd},
\[
\widehat\sigma=\frac12\,P_d\quad \text{on $(-1,1)$}.
\]
In particular, $\widehat\sigma$ is an ordinary function in a neighborhood of
zero and has no component supported at zero. Indeed, let
$g\in L^1(\mathbb R)$ be such that $\widehat g\in C_c(\mathbb R)$. Then for
all sufficiently large $X$, using $\widehat\sigma=P_d/2$ in a neighborhood
of zero, we have
\[
\frac1X\int_{\mathbb R}\sigma(x)g(x/X)\,dx
=\frac1{4\pi X}\int_{\mathbb R}P_d(t/X)\widehat g(-t)\,dt
=O(X^{-1}).
\]
Approximating $\mathbf1_{(0,1)}$ in $L^1(\mathbb R)$ by such functions $g$,
we obtain
\begin{equation}\label{sigma-mean-zero}
\frac1X\int_0^X\sigma(x)\,dx\to0,\quad X\to\infty.
\end{equation}

Put
\[
E=\bigcup_{k\ge1}\,(\tau_{2k-1},\tau_{2k}),\quad A(X)=|E\cap(0,X)|.
\]
On the positive half-line, $\sigma=-1$ on $E$ and $\sigma=1$ outside $E$
almost everywhere, hence
\[
\int_0^X \sigma(x)\,dx=X-2A(X)
\]
and \eqref{sigma-mean-zero} gives
\begin{equation}\label{E-density}
A(X)=\frac X2+o(X).
\end{equation}
For $s>0$, pairing consecutive terms in \eqref{eta-def} gives
\[
\eta_{\tau,d}(s) =s\int_E x^{-s-1}\,dx =s(s+1)\int_{\tau_1}^\infty
A(x)x^{-s-2}\,dx.
\]
By \eqref{E-density}, as $s\downarrow0$ the right-hand side tends to $1/2$.
Since $\eta_{\tau,d}$ is holomorphic at zero, for every dimension we obtain
\[
\eta_{\tau,d}(0)=\frac12.
\]

\subsection{Trivial zeros}
In addition, $\eta_{\tau,d}$ has the trivial zeros
\[
\eta_{\tau,d}(-d-2m)=0,\quad m\ge1.
\]
Indeed, at these points $\chi_\epsilon(s)=0$, while the remaining factor in
\eqref{eta-cont} is holomorphic.

\section{Additional consequences}\label{sec-final}

\subsection{Expansion for $1/\varphi$ and zero equilibrium}
From \eqref{phi-Theta}, \eqref{def-Theta}, we obtain
\[
\frac1{\varphi(z)} =\frac{z^{d+1}M(z)-2\epsilon z^d+4R_d(z)}{4a_d}.
\]
Substituting the series \eqref{def-M}, we find
\begin{equation}\label{reciprocal-series}
\frac1{\varphi(z)} =\frac{R_d(z)}{a_d} +\frac{z^{d+1+\epsilon}}{a_d}
\sum_{k=1}^{\infty}(-1)^k\,\frac{\tau_k^{1-\epsilon}}{z^2-\tau_k^2}.
\end{equation}
The series converges locally uniformly away from the poles for the same reason
as the series for~$M$ in \eqref{def-M}. Comparing the residues at
$z=\tau_k$, we obtain
\[
\frac1{\varphi'(\tau_k)} =\frac{(-1)^k\tau_k^{d+1}}{2a_d} =(-1)^kv_{d}\mathcal
C_d\tau_k^{d+1}.
\]
Equivalently,
\begin{equation}\label{derivative-identity}
\tau_k^{d+1}\varphi'(\tau_k)=2(-1)^ka_d.
\end{equation}
Differentiating the canonical product for $\varphi$ and comparing with
\eqref{derivative-identity}, we obtain the multiplicative equilibrium condition for the zeros
\begin{equation}\label{a-tau}
a_d=\tau_k^d\prod_{\substack{j\ge1,\;j\ne k}}
\Bigl|1-\frac{\tau_k^2}{\tau_j^2}\Bigr|.
\end{equation}

This condition gives a natural numerical algorithm valid in all dimensions.
For sufficiently large $J$, the zeros with $j>J$ can be described by an
asymptotic model
\[
\frac{\tau_j}{\pi}
=s_j+\frac{b_1}{s_j}+\frac{b_2}{s_j^3}+\cdots,
\quad s_j=j+\frac d2.
\]
After substituting this model, the infinite product over $j>J$ can be
evaluated explicitly as a ratio of finitely many gamma functions using
Euler's product formula for the gamma function. Thus the equilibrium condition reduces to a finite nonlinear
system for the initial zeros and the parameters of the asymptotic model, which
can be solved, for example, by Newton's method. The graph in
Subsection~\ref{subsec-asymp-d} was obtained in this way.

\subsection{Signs of the polynomial coefficients}
All coefficients of the even polynomial
\[
R_d(z)=\sum_{j=0}^{N+\epsilon}r_jz^{2j}
\]
from \eqref{intro-Rd} are positive. Indeed, using the canonical product
\eqref{phi-product}, put
\[
H(z)=\frac1{\varphi(z)} =\prod_{k\ge1}\Bigl(1-\frac{z^2}{\tau_k^2}\Bigr)^{-1}.
\]
In the disk $|z|<\tau_1$, expanding each factor into a geometric series shows
that the Taylor series of $H$ at zero has strictly positive even
coefficients. On the other hand, by \eqref{reciprocal-series},
\[
R_d(z)=a_dH(z)+O(z^{d+1+\epsilon}),\quad z\to0.
\]
Since $\deg R_d\le d-1+\epsilon$, all coefficients $r_j>0$.

By definition \eqref{intro-Kd},
\[
K_d(z)=\frac14\,z^{d+1}+2\sum_{n=0}^N
\biggl(\sum_{j=0}^nr_j\lambda_{n-j}\biggr)z^{2n+\epsilon}.
\]
Since \eqref{def-lambda} gives $\lambda_q<0$, while $r_j>0$, all coefficients
of $z^{2n+\epsilon}$, $0\le n\le N$, are negative. Therefore
\[
K_d(z)=\frac14\,z^{d+1}-\sum_{n=0}^Nk_nz^{2n+\epsilon},
\quad k_n>0.
\]
We also note that for $x>0$ this gives
\[
K_d(x)\le\frac{x^{d+1}}4.
\]
Hence
\[
V_d(x) =\frac{K_d(x)}{x^{d+1}} -\frac{R_d(x)^2}{x^{2d+2}}
-\frac{d^2-1}{4x^2}\le \frac14-\frac{d^2-1}{4x^2}, \quad x>0.
\]
By Sturm comparison, this estimate gives control of the number of zeros of
the corresponding equation and will be useful in studying the asymptotics as $d\to\infty$.

\subsection{Refined asymptotics of $\mathcal L^*(d)$}\label{subsec-asymp-d}
The asymptotics of the normalized Nikolskii constant $\mathcal L^*(d)$ from
\eqref{def-Lstar} as $d\to\infty$ will be studied in a separate paper. Based
on the results of the present paper, the correct exponential order will be
established there:
\[
\mathcal L^*(d)=2^{-d(1+o(1))},\quad d\to\infty.
\]

Numerical experiments indicate the more precise asymptotic relation
\begin{equation}\label{emp-form}
2^d\mathcal L^*(d)=\frac{\pi}{2}-cd^{-1/3}(1+o(1)), \quad c\approx0.777,
\end{equation}
as illustrated by the following graph:
\[
\includegraphics[scale=.6]{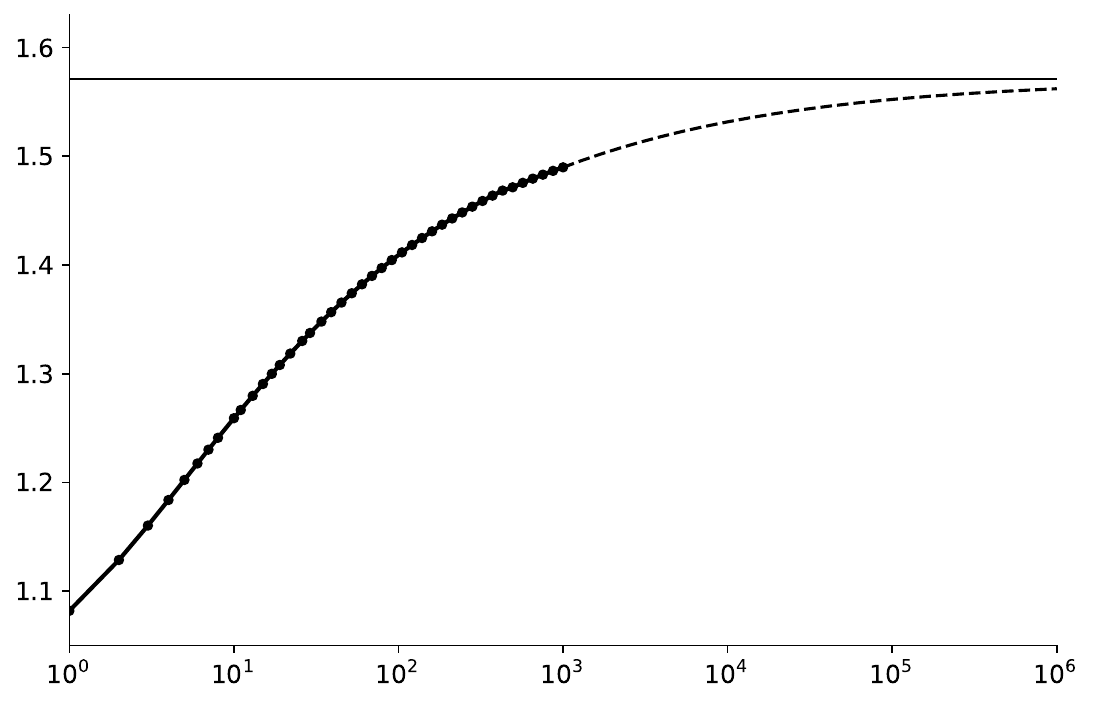}
\]
The solid line with dots shows the numerical values of
$2^d\mathcal L^*(d)$, while the dashed line shows their empirical
extrapolation given by \eqref{emp-form}. The horizontal line corresponds to the value
$\pi/2$.

\subsection*{Acknowledgments}
The author is grateful to the artificial intelligence model used in preparing
the paper for its assistance.

\end{document}